\documentclass[a4paper, 11pt]{article}

\usepackage[utf8]{inputenc}
\usepackage[T1]{fontenc}
\usepackage{fullpage}

\usepackage{comment}
\usepackage{amsmath, amsthm, amssymb}
\usepackage{graphicx}
\usepackage{enumerate}
\usepackage{authblk}
\usepackage{thmtools}
\usepackage{thm-restate}
\usepackage[colorlinks=true, citecolor=red]{hyperref}
\usepackage{float}

\usepackage[linesnumbered,ruled,onelanguage,vlined]{algorithm2e}

\usepackage{tikz}
\renewenvironment{abstract}
{\small\vspace{-1em}
\begin{center}
\bfseries\abstractname\vspace{-.5em}\vspace{0pt}
\end{center}
\list{}{
\setlength{\leftmargin}{0.6in}\setlength{\rightmargin}{\leftmargin}}\item\relax}
{\endlist}
\usepackage{comment}
\declaretheorem[name=Theorem, numberwithin=section]{theorem}
\declaretheorem[name=Lemma, sibling=theorem]{lemma}

\declaretheorem[name=Claim, sibling=theorem]{claim}

\def\cqedsymbol{\ifmmode$\lrcorner$\else{\unskip\nobreak\hfil
\penalty50\hskip1em\null\nobreak\hfil$\lrcorner$
arfillskip=0pt\finalhyphendemerits=0\endgraf}\fi}

\usetikzlibrary{calc}

\let\ge\geqslant
\let\leq\leqslant
\let\geq\geqslant

\title{A note on highly connected $K_{2,\ell}$-minor free graphs\thanks{This work was supported by ANR project GrR (ANR-18-CE40-0032)}}
\author[1]{Nicolas Bousquet}
\author[1]{Théo Pierron}
\author[2]{Alexandra Wesolek\thanks{The author is supported by the Vanier Canada Scholarship Program.}}

\affil[1]{Univ. Lyon, Université Lyon 1, CNRS, LIRIS UMR CNRS 5205, F-69621, Lyon, France}
\affil[2]{Department of Mathematics, Simon Fraser University, Burnaby, BC, Canada.}

\date{}

\begin{document}

\maketitle

\begin{abstract}
    We show that every $3$-connected $K_{2,\ell}$-minor free graph with minimum degree at least $4$ has maximum degree at most $7\ell$. As a consequence, we show that every 3-connected $K_{2,\ell}$-minor free graph with minimum degree at least $5$ and no twins of degree $5$ has bounded size. Our proofs use Steiner trees and nested cuts; in particular, they do not rely on Ding's characterization of $K_{2,\ell}$-minor free graphs.
\end{abstract}

\section{Introduction}
Over the last decades minor-free graph classes received considerable attention. Many computationally difficult problems become tractable on minor-free graph classes, especially if the excluded minor is planar. One of the simplest planar graphs is the complete bipartite graph $K_{2,\ell}$. A first step to understanding the difficulty of a problem is to consider it on the class of $K_{2,\ell}$-minor free graphs. The structure of $K_{2,\ell}$-minor free graphs has been widely studied, see e.g.~\cite{ChudnovskyRS11,Ding17+,EllinghamMOT16}. 

An important graph class parameter is the edge density. One can easily notice that a $K_{2,\ell}$-minor free graph $G$ may contain $\Omega(\ell \cdot n)$ edges since the disjoint union of $K_{\ell+1}$ is $K_{2,\ell}$-minor free. A similar example can be easily found for connected graphs. Chudnovsky, Reed and Seymour proved in~\cite{ChudnovskyRS11} that it is essentially tight since $K_{2,\ell}$-minor free graphs have at most $\frac{1}{2} (\ell + 1)(n - 1)$ edges, and they showed that this bound can be reached (proving a conjecture of Myers~\cite{Myers03}).

Ding~\cite{Ding17+} proposed a decomposition theorem for $K_{2,\ell}$-minor-free graphs (so far unpublished). As a corollary, he deduced that all the $5$-connected $K_{2,\ell}$-minor free graphs have bounded size, as well as $3$-connected $K_{2,\ell}$-minor free graphs of minimum degree at least $6$.  One can remark that these bounds are tight because of the strong product of an $n$-vertex cycle with an edge, which is $4$-connected and $5$-regular. Ding also obtained as a corollary that $3$-connected graphs with minimum degree $6$ have bounded maximum degree.

We reprove these results in a slightly stronger sense with simple self-contained proofs.  Namely, we show the following. Recall that two vertices are \emph{twins} if they have the same closed neighborhood (in particular, they are adjacent).

\begin{theorem}\label{thm:main}
All $3$-connected $K_{2,\ell}$-minor free graphs with minimum degree $5$ containing no twins of degree $5$ have bounded size.
\end{theorem}

 Observe that $5$-connected graphs, and $3$-connected graphs with minimum degree at least~$6$ fall in the scope of the theorem. Many practical problems such as Maximum Independent Set and Minimum Dominating Set can be reduced to twin-free graphs, which motivates their study. The theorem is tight because of the family of twin-free $4$-connected and $4$-regular graphs depicted in Figure~\ref{fig:neckK4}. The graphs in the family consist of two cycles $v_1 -v_2-\dots-v_n- v_1$ and $w_1 - w_2-\dots -w_n-w_1$ such that $v_i$ is connected to $w_i,w_{i+1}$ where $w_{n+1}:=w_1$. These graphs are indeed $K_{2,5}$-minor free. Indeed, suppose $A_1, A_2$ are two sets of edges whose contraction gives the two vertices of degree $5$ in $K_{2,5}$ and suppose $A_1,A_2$ are maximal. That means that $G\setminus A_1$ and $G\setminus A_2$ contain at most one component which is not a single vertex and $A_1,A_2$ induce a connected graph each.  Therefore $A_1,A_2$ induce a path or a cycle when we restrict them to the inner and outer cycles. But now, deleting the edges between these two sets yields a graph with a cut of size at most $4$ separating them. A $K_{2,5}$-minor would imply that there are $5$ vertex-disjoint paths between $A_1$ and $A_2$ with non-empty interior, which is not possible.
\begin{figure}[!ht]
\centering
\begin{tikzpicture}[thick, every node/.style={draw, circle}]
\node (a) at (90:1) {};
\node (b) at (135:1) {};
\node (c) at (180:1) {};
\node (d) at (225:1) {};
\node (e) at (270:1) {};
\node (f) at (315:1) {};
\node (g) at (0:1) {};
\node (h) at (45:1) {};
\node (aa) at (90:2) {};
\node (bb) at (135:2) {};
\node (cc) at (180:2) {};
\node (dd) at (225:2) {};
\node (ee) at (270:2) {};
\node (ff) at (315:2) {};
\node (gg) at (0:2) {};
\node (hh) at (45:2) {};
\draw (a) -- (b) -- (c) -- (d) -- (e) -- (f) -- (g) -- (h);
\draw[dotted] (h) -- (a) -- (hh) -- (aa);
\draw (a) -- (aa) -- (bb) -- (cc) -- (dd) -- (ee) -- (ff) -- (gg) -- (hh);
\draw (aa) -- (b) -- (bb) -- (c) -- (cc) -- (d) -- (dd) -- (e) -- (ee) -- (f) -- (ff) -- (g) -- (gg) -- (h) -- (hh);
\end{tikzpicture}
\caption{An infinite family of twin-free $4$-connected and $4$-regular graphs without $K_{2,5}$-minor.}
\label{fig:neckK4}
\end{figure}
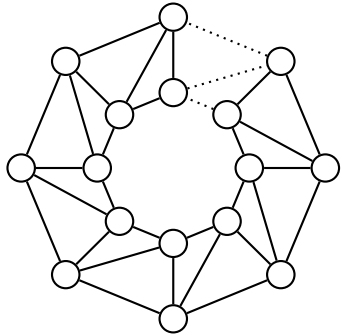

Our proof of Theorem~\ref{thm:main} consists of two parts. The first, of independent interest, consists in proving that $3$-connected $K_{2,\ell}$-minor free graphs of minimum degree $4$ have maximum degree $O(\ell)$ (it strengthens a result of Ding~\cite{Ding17+} who only proved it for minimum degree $6$). The second consists in constructing a $K_{2,\ell}$-minor when there are too many cuts between two given vertices.

Chudnovsky, Reed and Seymour gave in~\cite{ChudnovskyRS11} the sketch of a proof that $3$-connected $K_{2,\ell}$-minor free graphs contain at most $\frac 52 n + c(\ell)$ edges for a large constant $c$ (not explicit). Their technique, based on an adaptation of a tree decomposition argument of~\cite{oporowski1993typical} is completely different from ours. Even if weaker, we can easily derive from our proof that $3$-connected $K_{2,\ell}$-minor free graphs with minimum degree $5$ have at most $\frac 52 n+c(\ell)$ edges.

\section{Maximum degree}

Let $s,t$ be two vertices of the graph $G$. An \emph{$st$-cut} is a subset of vertices $C$ not containing $s$ nor $t$ and whose deletion separates $s$ from $t$. We say that $C$ is a \emph{$k$-cut} if $|C|=k$ and we denote by $\eta(s,t)$ the \emph{$st$-connectivity}, that is the minimum size of an $st$-cut. 
In~\cite{Ding17+}, Ding proved that $K_{2,\ell}$-minor free graphs that are $4$-connected or $3$-connected with minimum degree $6$ have bounded degree. We slightly improve these results (with a short proof):

\begin{lemma}
\label{lem:degree}
Every $3$-connected $K_{2,\ell}$-minor free graph $G$ with minimum degree $4$ has maximum degree at most $7\ell$.
\end{lemma}
\begin{proof}
Assume by contradiction that there exists a $3$-connected  $K_{2,\ell}$-minor free graph $G$ of minimum degree $4$ with a vertex $x$ of degree at least $7\ell$. Let $X=N(x)$ and $G'=G \setminus x$. Let $T$ be a Steiner tree of $X$ in $G'$ (a subtree of $G'$ spanning all vertices in $X$ with no leaf outside of $X$) maximizing the number of leaves $L$ in $X$. Note that if $T$ has at least $\ell$ leaves, then contracting all internal vertices of $T$ yields a $K_{2,\ell}$ minor in $G$. 

Therefore, we may assume that $T$ has at most $\ell-1$ leaves. Hence $T$ has less than $\ell$ branching vertices (that is vertices of degree at least $3$ in $T$). In particular, at least $6\ell$ vertices of $X$ have degree at most $2$ in $T$. They then belong to the set of paths of $T$ minus its branching vertices. Note that there are less than $2\ell$ such paths since $T$ has less than $\ell$ branching vertices and at most $\ell-1$ leaves. For each such path containing the vertices $u_1,u_2,\ldots,u_s$ of $X$ in that order, we say that $u_i$ is a \emph{strictly internal} vertex if $i \notin \{ 1,s \}$, and that it is \emph{even} (resp. \emph{odd}) if furthermore $i$ is even (resp. odd). Observe that on each of the less than $2\ell$ paths there are at most two non-strictly internal vertices and the number $e$ of even vertices is at least the number $o$ of odd vertices. To summarize, $X$ contains at most $\ell-1$ branching vertices, at most $2\times 2\ell$ non-strictly internal vertices and $e+o$ strictly internal vertices, hence we have $7\ell \leqslant |X|\leqslant \ell-1+o+e+2\cdot 2\ell\leqslant 5\ell-1+2e$ so $e>\ell$.

Let $u$ be an even vertex, and denote by $v,w$ the previous and next vertices of $X$ on the path containing $u$ in $T$ minus its branching vertices. For $a,b$ in $T$, we denote by $P_{ab}$ the subpath of $T$ connecting $a$ and $b$.

First observe that there is no path $Q$ from an internal vertex $v'$ of $P_{uv}$ to $T \setminus (P_{uv} \cup L)$ whose internal vertices are in $V(G') \setminus T$. Indeed, otherwise, we can replace either $P_{uv'}$ or $P_{vv'}$ by $Q$ in $T$. This yields another Steiner tree of $X$ in $G'$ with one more leaf than $T$ (namely $u$ or $v$), a contradiction. 

\begin{claim}
  There exists a path $Q_u$ from a vertex of $P_{vw}\setminus\{v,w\}$ to a vertex of $T\setminus P_{vw}$ whose internal vertices are in $V(G')\setminus T$.\end{claim}

\begin{proof}

Assume first that $Q:=P_{uv} \setminus \{ u,v\}$ contains a vertex. Since $G$ is $3$-connected and $x$ is not adjacent to any vertex of $Q$, $\{u,v \}$ is not a cut and then there exists a path $R$ from a vertex $z$ of $Q$ to $T \setminus P_{uv}$ whose internal vertices are in $V(G') \setminus T$. 
If $R$ has its other endpoint in $P_{uw} \setminus \{ u \}$, then adding $R$ and deleting the edges of $P_{zu}$ leaves a new Steiner tree with more leaves, a contradiction. So $R$ has its other endpoint in $T \setminus P_{vw}$ which completes the proof in that case.

So we can assume that $vuw$ is an induced $P_3$. Since $G$ has minimum degree at least $4$, $u$ has a neighbor $u' \notin\{ v,w\}$ not in $G'$. If $u' \in T$, the conclusion follows. So $u' \notin T$ and let $C$ be the connected component of $u'$ in $V(G') \setminus T$. Note that $C$ is not adjacent to $x$. Therefore $C$ must be adjacent to a vertex $t$ of $T \setminus u$ otherwise $u$ would be a cut vertex. If $t \notin\{v,w\}$, the conclusion follows. Note that $C$ cannot be adjacent to exactly one of $\{v,w\}$, otherwise $\{u,v\}$ or $\{u,w\}$  would be a $2$-cut. So $C$ is adjacent to both $v$ and $w$, and then adding to $T$ a tree connecting $v,w,u$ in $C$ and deleting the edges $uv$ and $uw$ yields a Steiner tree with an additional leaf, again a contradiction. 

 \end{proof}

Now, adding $Q_u$ to $T$ creates a cycle containing either $v$ or $w$, say $v$ by symmetry. Now we remove the edges of this cycle on $P_{uv}$, and the result is a Steiner tree of $X$ in $G'$ where $v$ is a leaf. By maximality of $T$, this means that the endpoint of $Q_u$ in $T$ must be a leaf. 

Since there are at least $\ell$ even vertices and at most $\ell-1$ leaves in $T$, at least two paths $Q_u$ and $Q_{u'}$ must end on the same leaf of $T$. Define $v$ (resp. $v'$) as the vertex of $X$ preceding or following $u$ (resp. $u'$) in the path containing $u$ (resp. $u'$) in $T$ minus its internal vertices, such that $v$ (resp. $v'$) lies in the unique cycle of $T\cup Q_u$ (resp. $T\cup Q_{u'}$).

Now removing the edges of $P_{uv}$ and $P_{u'v'}$ from $T$ and adding a spanning tree of $Q_u\cup Q_{u'}\cup R$, where $R$ is the shortest subpath of $P_{vw}$ containing $u$ and the endpoints of $Q_u$ and $Q_{u'}$ in $P_{vw}$, yields a Steiner tree where $v$ and $v'$ are leaves, hence it has one more leaf than $T$, a contradiction.
\end{proof}

Note that both hypotheses are required, as shown by the following examples with arbitrary large maximum degree:
\begin{itemize}
\item Wheels have minimum degree $3$ and the central vertex has arbitrary large degree. Wheels are $3$-connected and $K_{2,4}$-minor-free.
\item Starting from an even wheel $u,u_1,\ldots,u_{2n}$, and, for each $i\in[1,n]$, adding three vertices $x_i,y_i,z_i$ such that $\{u_{2i-1},u_{2i},x_i,y_i,z_i\}$ induces a $K_5$ yields a $2$-connected graph where $u$ has arbitrary high degree. This graph has minimum degree $4$ and is $K_{2,8}$-minor-free. Indeed, if it contains a $K_{2,8}$-minor, let $C$ and $C'$ be connected subsets of vertices that are contracted to create the vertices of degree $8$. Without loss of generality assume that $C$ does not contain $u$. Since $C$ is connected, it induces a path on the cycle $u_1,\dots,u_{2n}$ or contains the cycle (and note that $C$ contains a vertex from the cycle, otherwise there is a cut of size at most 2 separating $C$ from $C'$).  Observe further that for every $i\in[1,n]$, we can assume that $C$ contains either all of $\{x_i,y_i,z_i\}$ or none of them. Moreover, since $C$ is connected, there are at most two indices $i$ and $j$ such that $C$ contains only one vertex in $\{u_{2i-1},u_{2i}\}$ and in $\{u_{2j-1},u_{2j}\}$. Now, when removing the edges between $C$ and $C'$, we get a graph with a cut of size at most $7$ (using vertices among $u,x_i,y_i,z_i,x_j,y_j,z_j$) separating $C$ from $C'$, which is impossible. (Note that this can be extended so that the graph has minimum degree $\delta$ and is $K_{2,2\delta}$-minor-free.)
\end{itemize}

\section{Proof of Theorem~\ref{thm:main}}

Let $s,t$ be two vertices of the graph $G$. For every $st$-cut $C$, we denote by $S_C$ the connected component of $s$ in $G\setminus C$.
A sequence of cuts is \emph{nested} if for every two cuts $C,C'$ in the sequence, 
$S_C \subseteq S_{C'}$ or $S_{C'} \subseteq S_C$. Finally, we say a sequence of cuts is \emph{$2$-nested} if the cuts are nested and the distance between any two cuts is at least $2$, i.e. they are vertex disjoint and there is no edge between distinct cuts.
All these notions can be extended when $s$ and $t$ are replaced by sets of vertices. Note that in particular we consider that if $S,T$ are disjoint sets of vertices, an $ST$-cut is disjoint from $S\cup T$.
The proof of Theorem~\ref{thm:main} relies on the following lemma, which shows that if we have a large enough $2$-nested sequence of $st$-cuts, then $G$ admits a $K_{2,\ell}$-minor.

\begin{lemma}
\label{cl:nestedkcuts_3conn}
    Suppose that $G$ is $3$-connected, has minimum degree $5$ and no twins of degree~$5$. Let $S,T$ be two subsets of vertices inducing connected subgraphs of $G$. If $G$ contains a $2$-nested sequence of more than $\ell \cdot {\eta(S,T)\choose 2}$ $ST$-cuts of size $\eta(S,T)$ then $G$ admits a $K_{2,\ell}$-minor.
\end{lemma}

\begin{proof}
Let $\eta:=\eta(S,T)$, $d:=\ell \cdot {\eta\choose 2}$ and denote by $C_1,\ldots, C_d$ the $2$-nested sequence of $\eta$-cuts. Consider $\eta$ internally disjoint $ST$-paths $P_1,\ldots,P_{\eta}$ and write $P=P_1\cup\cdots\cup P_{\eta}$. For each $j\in[1,d-1]$, denote by $Y_j$ the set $S_{C_{j+1}}-(S_{C_j}\cup C_j)$ of vertices between the cuts $C_j$ and $C_{j+1}$. Observe that, for each $j\in[1,d-1]$, exactly one of the following may happen:
\begin{itemize}
    \item[(1)] for some $a\neq b$, there exists a path of length at least two between $P_a\cap (Y_j\cup C_j\cup C_{j+1})$ and $P_b\cap (Y_j\cup C_j\cup C_{j+1})$ with internal vertices in $Y_j\setminus (P_a\cup P_b)$.
    \item[(2)] $Y_j\setminus P$ is not empty and (1) does not hold. In particular, each connected component of $Y_j\setminus P$ is adjacent to exactly one path.
    \item[(3)] $Y_j\setminus P$ is empty and (1) does not hold.
\end{itemize} 
We choose $P_1,\ldots,P_\eta$ such that the number of times (1) holds is maximized, and with respect to that, the number of times (2) holds is maximized.  Fix $j\in[1,d-1]$.

If (2) holds, then there exists a component $C$ in $Y_j \setminus P$ which has only neighbours on some path, say $P_1$. Let $v_1,v_2,\dots, v_r$ be its neighbours on $P_1$ in order (note that $r\geqslant 3$ otherwise $\{v_1,v_2\}$ is a $2$-cut in $G$). Since $C$ is connected, there exists a path $P_C$ in $C$ from a neighbour of $v_1$ to a neighbour of $v_r$. We replace $P_1$ by the path $P'_1$ following $P_1$ up to $v_1$, then taking $P_C$, and following $P_1$ starting from $v_r$. Since $G$ is $3$-connected, deleting $v_1,v_r$ does not disconnect the graph. So there is a path from $C$ to some other path in $(Y_j\cup C_j\cup C_{j+1}) \setminus \{ v_1,v_r \}$, say $P_2$, and this path does not intersect $P_3,\dots,P_{\eta}$. Since $C$ does not have neighbors on any $P_i$ for $i>1$, this path must go through some $v_i$ with $i\notin\{1,r\}$. Now the component of $v_i$ in $Y_j\setminus (P_1'\cup P_2\cup\cdots\cup P_\eta)$ has neighbors both in $P'_1$ and $P_2$, hence satisfies (1), a contradiction with our choice of $P$. 

Assume now that (3) holds. Since the sequence is $2$-nested, $Y_j\cap P_1$ is not empty. Let $v$ be the first vertex of $P_1$ in $Y_j$. Note that $Y_j\cap P_1$ is induced, otherwise we can replace it by a shorter subpath and end up in case (1) or (2), a contradiction with our choice of $P$. Therefore, since $v$ has degree at least $5$, $v$ must have at least three neighbors outside of $P_1$. They should lie on exactly one other path since $Y_j\setminus P$ is empty and (1) does not hold, say $P_2$ by symmetry. Let $w$ be the first neighbor of $v$ in $P_2\cap Y_j$. Note that $w$ does not belong to $C_j$ nor $C_{j+1}$. Observe that $v, w$ have only neighbours on $P_1\cup P_2$, and denote by $w_1,\ldots,w_t$ the neighbors of $v$ in $P_2$ (following their order in $P_2$) and by $v_1,\ldots,v_s$ the neighbors of $w$ in $P_1$. Note that $s,t \ge 3$ and that $v\in\{v_1,v_2\}$ and $w\in\{w_1,w_2\}$ depending on whether $v_1$ or $w_1$ lie in $C_j$ or not. 

Suppose that for some $i>2$, $v_i\notin N[v]$. Let $P_1^{-v}$ be the subpath of $P_1$ where we delete all vertices after $v$ and let $P_1^{+v_i}$ be the subpath of $P_1$ where we delete all vertices before $v_i$. Similarly, let $P_2^{-w}$ be the subpath of $P_2$ where we delete all vertices after $w$ and let $P_2^{+w_t}$ be the subpath of $P_2$ where we delete all vertices before $w_t$. Now replacing $P_1$ by $P_1^{-v}P_2^{+w_t}$ and $P_2$ by $P_2^{-w}P_1^{+v_i}$ gives a new collection of $\eta$ internally disjoint paths from $S$ to $T$ where the total length of the paths in the section $j$ has been reduced. So (3) does not hold anymore and the number of sections satisfying (1) or (2) has increased, a contradiction with the choice of $P$.

Suppose now that $v_i\in N[v]$ and by symmetry $w_i \in N[w]$ for all $i>2$. Recall that $P_1$ is induced, hence $vv_i$ cannot be an edge for $i\geqslant 4$. Therefore, $s=3$ and $v=v_2$. By symmetry, we also get $t=3$ and $w=w_2$. Now, if $vv_1$ is not an edge, then we can again replace $P_1$ by $P_1^{-v_1}P_2^{+w}$ and $P_2$ by $P_2^{-w_1}P_1^{+v}$ and get a contradiction. Hence $vv_1$ is an edge and by symmetry $ww_1$ is an edge. Therefore, $v$ and $w$ have both degree $5$ and are twins, which violates the hypothesis on $G$.

Therefore, we may assume that $(1)$ holds for every $j\in[1,d-1]$. By the pigeonhole principle, there is $a,b\in[1,\eta]$ with $a\neq b$ and at least $\ell$ paths of length at least $2$ between $P_{a}$ and $P_{b}$ with internal vertices outside of $P_{a}\cup P_{b}$. Since these internal vertices lie in pairwise disjoint sets $Y_j$, they are internally disjoint. Therefore contracting the internal vertices of these paths into $\ell$ vertices and $P_a$ and $P_b$ into two vertices yields a $K_{2,\ell}$ minor in $G$. 

\end{proof}

We now conclude this section with the proof of Theorem~\ref{thm:main}, obtained by combining Lemmas~\ref{lem:degree} and~\ref{cl:nestedkcuts_3conn}.

\begin{proof}[Proof of Theorem \ref{thm:main}]
Let $d= \ell^3/2$, $n_\ell \geq (7\ell)^{2\cdot (d+1)^\ell}$ and $G$ be a graph on $n_\ell$ vertices satisfying the hypothesis of the theorem. We may assume that $G$ has maximum degree at most $7\ell$ for otherwise we get a $K_{2,\ell}$-minor by Lemma~\ref{lem:degree}. Therefore, there must be two vertices $s,t$ at distance at least $2\cdot (d+1)^\ell$ in $G$. Let us denote by $\eta(s,t)$ the $st$-connectivity.

Let $V_i$ be the vertices at distance $i$ to $s$ and let $C_i\subseteq V_i$ be an inclusion-wise minimal $st$-cut. In $V \setminus C_i$, denote by $S_i$ (resp. $T_i$) the set of vertices in the connected component of $s$ (resp. $t$). 

Let $k$ be the maximum integer such that $k\leqslant \ell$ and there exist $i>0$ and $j$ such that $j-i\geqslant 2\cdot (d+1)^{\ell-k}$ and $\eta(C_i,C_j)\geq k$. If $k=\ell$, we obtain a $K_{2,\ell}$-minor by contracting $S_i$ and $T_j$ into two vertices and contracting the internal vertices of each of the $\ell$ pairwise disjoint paths of length at least $2$ between $S_i$ and $T_j$ (and possibly deleting some edges).

Hence we can assume that $k\leq \ell-1$. Let $i,j$ be the corresponding indexes and let $d'=2\cdot(d+1)^{\ell-k-1}$. We consider the $d+1$ cuts $C_i, C_{i+d'}, C_{i+2d'}, \dots, C_{i+d\cdot d'}$ (which are pairwise at distance at least $d'$). By maximality of $k$, we cannot have $\eta(C_i,C_j)>k$, nor $\eta(C_{i+i'd'},C_{i+(i'+1)d'})>k$ for any $0  \leq i'\leq d-1$. Hence $\eta(V_{i+i'd'},V_{i+(i'+1)d'})=k=\eta(C_i,C_j)$ for each $i'\leq d-1$. But then there exists a cut of size $k$ between $C_{i+i'd'},C_{i+(i'+1)d'}$ for each $0 \leq i'\leq d-1$. These $d$ $k$-cuts between $C_i$ and $C_j$ are $2$-nested, hence by Lemma~\ref{cl:nestedkcuts_3conn}, $G$ then contains a $K_{2,\ell}$-minor, which concludes the proof. 
\end{proof}

One can easily adapt the proofs of Lemma~\ref{cl:nestedkcuts_3conn} and Theorem~\ref{thm:main} to show that every large enough $3$-connected graph $G$ with minimum degree $5$ either contains a $K_{2,\ell}$-minor or contains the $2\times 4$ king's graph\footnote{A $(k,\ell)$-king's graph is a square grid on $(k+1)\times(\ell+1)$ vertices with diagonals.} as an induced subgraph where the middle vertices have degree $5$ in $G$ (the current proof of Lemma~\ref{cl:nestedkcuts_3conn} only ensures the existence of a $2\times 3$ king's graph). Contracting the $K_4$ in the middle into a single edge to obtain a $2\times 3$ king's graph as in Figure \ref{fig:kingcontraction} keeps the minimum degree and the connectivity of the graph. Moreover, it removes two vertices and five edges from the graph. Therefore, applying this transformation while possible yields the following weakening of the result by Chudnovsky et al.~\cite{ChudnovskyRS11}:  every $3$-connected $K_{2,\ell}$-minor free graph of minimum degree $5$ contains at most $\frac 52 n +c(\ell)$ edges.
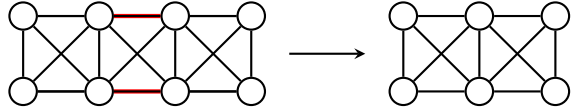
\begin{figure}[!ht]
    \centering
    \begin{tikzpicture}[thick, every node/.style={draw, circle}]
    \foreach \i in {1,2,3,4}
    {
 \node (a\i) at (\i,1) {};
 \node (b\i) at (\i,0) {};
 \draw (a\i)--(b\i);
    }
       \draw (a1)--(a2)--(a3)--(a4);
\draw (b1)--(b2)--(b3)--(b4);
  \foreach \i/\j in {1/2,2/3,3/4}
{
\draw  (a\i) --(b\j);
\draw  (b\i) --(a\j);
}
\draw[ultra thick,red] (a2)--(a3);
\draw[ultra thick,red] (b2)--(b3);
\draw [-stealth](4.5,0.5) -- (5.5,0.5);
    \foreach \i in {6,7,8}
    {
 \node (a\i) at (\i,1) {};
 \node (b\i) at (\i,0) {};
 \draw (a\i)--(b\i);
    }
       \draw (a1)--(a2)--(a3)--(a4);
\draw (b1)--(b2)--(b3)--(b4);
  \foreach \i/\j in {6/7,7/8}
{
\draw  (a\i) --(b\j);
\draw  (b\i) --(a\j);
\draw (a\i) -- (a\j);
\draw (b\i) -- (b\j);
}
    \end{tikzpicture}
    \caption{Contracting the red thick edges reduces a $2\times 4$ king's graph to a $2\times 3$ king's graph.}
    \label{fig:kingcontraction}
\end{figure}

\vspace{-.5cm}
\bibliographystyle{abbrv}

\end{document}